\documentclass[reqno, 10pt]{amsart}
\usepackage[centertags]{amsmath} \usepackage{amsfonts}
\usepackage{amssymb} \usepackage{amsthm}
\usepackage{amscd}
\usepackage{newlfont}
 \DeclareMathOperator{\soc}{soc}
\DeclareMathOperator{\aut}{Aut}

 \DeclareMathOperator{\frat}{Frat}

\DeclareMathOperator{\core}{Core}

\DeclareMathOperator{\alt}{Alt}
\DeclareMathOperator{\End}{End} 
 \DeclareMathOperator{\diag}{Diag}

\newcommand{\ol}{\overline}

\newcommand{\C}{ c} 
\newcommand{\D}{ d} 
\newcommand{\E}{e} 

\newcommand{\F}{\mathbb F}
\newcommand{\Z}{\mathbb Z}
\newcommand{\N}{\mathbb N}

\newtheorem{thm}{Theorem}

 \newtheorem{lemma}[thm]{Lemma}
\newtheorem{prop}[thm]{Proposition}

\numberwithin{equation}{section}

\renewcommand{\footnote}{\endnote}
\newcommand{\ignore}[1]{}\makeglossary

\begin{document}
\bibliographystyle{amsplain}

\title{Characterization of finitely generated infinitely iterated wreath products} 
\author{Eloisa Detomi and Andrea Lucchini}

\address{  Dipartimento di Matematica Pura ed Applicata,
Universit\`a di Padova,
Via Trieste 63, 
 35121 Padova, Italy.}
\email{detomi@math.unipd.it, lucchini@math.unipd.it}

\subjclass{Primary 20F05; Secondary 20E18 20E22.}
\keywords{Generation; wreath product; probability}

\begin{abstract}
Given a sequence of $(G_i)_{i \in \N}$ of finite transitive groups of degree $n_i$, let $W_\infty$ be the inverse limit of the iterated permutational wreath products $G_m \wr \cdots \wr G_2 \wr G_1$. We prove that $W_\infty$ is (topologically) finitely generated if and only if $\prod_{i=1}^{\infty} (G_i/G_i')$ is finitely generated and 
 the growth of the minimal number of generators of $G_i$  is bounded by  $ d \cdot n_1 \cdots n_{i-1}$ for a constant $d$. Moreover we give a criterion to decide whether $W_\infty$ is  positively finitely generated.  
\end{abstract}

\maketitle

\section{Introduction}
Let  $(G_i)_{i \in \N}$ be a sequence  of finite transitive permutation groups  of degree $n_i$ 
 and let  $W_m=G_m \wr \cdots \wr G_2 \wr G_1$  be the iterated (permutational) wreath product 
of the first $m$ groups. 
 The infinitely iterated wreath product is the inverse limit 
 $$W_\infty = \varprojlim_m W_m= \varprojlim_m (G_m \wr \cdots \wr G_2 \wr G_1).$$

In a recent paper Bondarenko \cite{bond}  studies some sufficient conditions on the sequence
 $(G_i)_{i \in \N}$ to get that the profinite group $W_\infty$ is (topologically) finitely generated:
 under the conditions that the minimal number of generators $d(G_i)$ of $G_i$ is bounded by a constant $d$
 and  $\prod_{i=1}^{\infty}  (G_i/G'_i)$ is finitely generated, using techniques from branch groups, 
 he  produces a finitely generated dense subgroup of $W_\infty$.

Since $\prod_{i=1}^{\infty}  (G_i/G'_i)$ is a homomorphic image of $W_\infty$,
 the second condition is clearly also a necessary condition: if  $W_\infty$ is  generated as a profinite group by $d$ elements, then  $d(\prod_{i=1}^{\infty}  (G_i/G'_i)) \leq d$.

Another necessary condition comes from the observation that  if $K$ is a finite permutation group of degree $n$ and $H$ is finite, then   $d(H) \leq  n \cdot  d(H \wr K)$ (see the remark at the beginning of section 5).   Since $W_i=G_i \wr W_{i-1}$ where $W_{i-1}$ is a permutation group of degree $ n_1 n_2 \cdots n_{i-1}$, it follows that if   $W_\infty$ is finitely generated by $d$ elements,  then
   $d(G_i) \leq d \cdot n_1 n_2 \cdots n_{i-1}$ for every $i>1$.

   The main result of this paper is   that these two necessary conditions are also sufficient. 
   \begin{thm}\label{main} Let $(G_i)_{i\in \N}$ be a sequence of  transitive permutation  groups of degree $n_i$.
   The inverse limit $W_\infty$ of the iterated wreath products $G_m \wr \cdots \wr G_2 \wr G_1$ is finitely generated if and only if
\begin{enumerate}
	\item $\prod_{i=1}^{\infty} (G_i/G_i')$ is finitely generated,
	\item there exists an integer $d$ such that $d(G_i) \leq d \cdot n_1 \cdots n_{i-1}$ for every $i>1$.
\end{enumerate}
   \end{thm}

 Actually, we prove that there exists an absolute constant $k_0$ such that
\begin{equation*}
 d(W_\infty) \leq \max (d+2, d(W_{i_0}))+d\left(\prod_{i=1}^{\infty} (G_i/G_i')\right).
\end{equation*}
 where $i_0$ is the first index such that   $n_1 \cdots n_{i_0-1}\geq \log_{60} k_0$.
 Indeed $k_0$ is the smallest positive integer with the property: if a finite group $L$ has a unique minimal normal subgroup $N$
 and $|N|\geq k_0$, then $P_L(d)\geq \frac 1 2 P_{L/N}(d)$ for each $d\geq 2$, where $P_L(d)$ (resp. $P_{L/N}(d))$ denotes the probability of generating
 $L$ (resp. $L/N$) with $d$ elements. The existence of such a constant is ensured by the main theorem in \cite{fiore}.
 On the other hand we conjecture that for every $d\geq 2$ and every monolithic group $L$ with socle $N$
\begin{equation}\label{53}P_L(d)\geq \frac {53}{90} P_{L/N}(d)\end{equation}
(the equality holds if $L=\alt(6)$ and $d=2$).
If this were true, our result would become 
\begin{equation*}
d(W_\infty)\leq \max (d+2, d(G_{1}))+d\left(\prod_{i=1}^{\infty} (G_i/G_i')\right).
\end{equation*}
For example, the inequality (\ref{53}) is satisfied if the socle of $N$ is a direct power of alternating or sporadic simple groups
\cite{emil}: this implies that if every  non-abelian composition factor in the  $G_i$'s is alternating or sporadic, then $d(  W_\infty) \leq \max (d+2, d(W_{1}))+d(\prod_{i=1}^{\infty} (G_i/G_i')).$

The proof of Theorem \ref{main} relies on a generalization to the \lq\lq non-soluble'' case of some
 results in   \cite{andrea1} and \cite{andrea2}.  In that papers the author  considered the  generation of the wreath product  $W=H \wr K$ of two finite permutation groups $H$ and $K$ and a formula was found for $d(W)$  in the case where $H$ is soluble.
Later, in \cite{dv-l}, the minimal number of generators of a group $G$ was connected to some special homomorphic images of $G$
whose behavior can be studied with the help of an equivalence relation among the 
chief factors of $G$ 
 (see section \ref{gen-crowns} for more details). 
 Using these new techniques, we are able to control the ``non-abelian'' part of the problem and to produce a formula for $d(W)$ whenever the degree of $K$  is large enough.

Infinitely iterated wreath products appear in  literature with several motivations. For example they can be viewed as automorphism groups of suitably constructed rooted trees and play a relevant role in the study of self-similar groups (see e.g. \cite{grig1}, \cite{grig2}). Moreover, they provide a useful tool to construct examples and counterexamples in the context of profinite groups (see e.g. \cite{neumann}, \cite{seg}, \cite{just infinite}). 
 Bhattacharjee \cite{bhat}  and  Quick \cite{Q1} \cite{Q_2} considered   wreath products of non-abelian simple groups with transitive action and proved that 
 their inverse limit
 is generated by $2$ elements even with positive probability. 
Recall that a profinite group $G$ may be viewed as a probability space with respect to the normalized Haar measure
and that $G$ is called positively finitely generated (PFG) if for some $k$ a random $k$-tuple generates $G$ with positive
probability.
  From the papers of Bhattacharjee and Quick, it follows that an infinitely iterated wreath product of transitive groups $G_i$'s
 is PFG when every $G_i$ is a nonabelian simple group.
 However in \cite{just infinite} an example is given of an  infinitely  iterated wreath
product of transitive groups that is 2-generated but non PFG.

 In Proposition \ref{pfg}, 
 with the help of a result by Jaikin-Zapirain and Pyber \cite{pfg-j-p},
we will obtain a criterion that makes it possible to decide whether $W_\infty$ is PFG from information on the structure of
the transitive groups $G_i$'s and their degree $n_i$'s.


\section{Generating crown-based powers}\label{gen-crowns} 
 
Let $L$ be   a monolithic primitive group
 and let $A$ be its unique minimal normal subgroup. For each positive integer $k$,
 let $L^k$ be the $k$-fold direct
product of $L$.
 The \emph{crown-based power} of $L$ of size  $k$ is the subgroup $L_k$ of $L^k$ defined by
 $$L_k=\{(l_1, \ldots , l_k) \in L^k  \mid l_1 \equiv \cdots \equiv l_k \ {\mbox{mod}} A \}.$$  
Equivalently, $L_k=A^k \diag L^k$. 

Let, as usual,  $d(G) $ denote the minimal number of generators of a finite group $G$. In \cite{dv-l} it is proved that for every finite group $G$ there exists a monolithic group $L$ and a homomorphic image $L_k$ of $G$ such that 
\begin{enumerate}
\item $ d(L/\soc L) < d(G) $ 
\item $d(L_k) =d(G).$ 
\end{enumerate} 
 An $L_{k}$ with this property will be called a \emph{generating crown-based power} for $G$. 
 In \cite{dv-l} it is explained how  $d(L_{k})$ can be computed in terms of $k$ and the structure of $L$.  
A key ingredient when one wants to determine $d(G)$ from the behavior of 
the crown-based power  homomorphic images of $G$  is to evaluate for each monolithic group $L$  the maximal $k$ such that    $L_{k}$ is a   homomorphic image. 
  This  integer $k$ 
 comes from an  equivalence relation among the chief factors of $G$. More generally, following  \cite{paz},  we say that  
two irreducible $G$-groups $A$ and $B$  are  \emph{$G$-equivalent} and we put $A \sim_G B$, if there is an
isomorphism $\Phi: A\rtimes G \rightarrow B\rtimes G$ such that the following diagram commutes:

\begin{equation*}
\begin{CD}
1@>>>A@>>>A\rtimes G@>>>G@>>>1\\
@. @VV{\phi}V @VV{\Phi}V @|\\
1@>>>B@>>>B\rtimes G@>>>G@>>>1
\end{CD}
\end{equation*}

\
Note that two $G$\nobreakdash-isomorphic
$G$\nobreakdash-groups are $G$\nobreakdash-equivalent. In the particular case where   $A$ and $B$ are abelian the converse is true: 
if $A$ and $B$ are abelian and $G$\nobreakdash-equivalent, then $A$
and $B$ are also $G$\nobreakdash-isomorphic.
 It is proved that two  chief factors $A$ and $B$ of $G$ are  $G$-equivalent if and only if  either they are  $G$-isomorphic between them or there exists a maximal subgroup $M$ of $G$ such that $G/\core_G(M)$ has two minimal normal subgroups $N_1$ and $N_2$ 
$G$-isomorphic to $A$ and $B$ respectively. For example, the  minimal normal subgroups of $L_k$ are all $L_k$-equivalent.

  Let $A=X/Y$ be a chief factor of $G$. A complement $U$ to $A$ in $G$ is a subgroup $U $ of $G$ such that $UX=G$ and $U \cap X=Y$. We say that   $A=X/Y$ is Frattini if  $X/Y$ is contained in the Frattini subgroup of $G/Y$; this is equivalent to say that $A$ is abelian and there is no complement to $A$ in $G$. 
The  number $\delta_G(A)$  of non-Frattini chief factors $G$-equivalent to $A$   in any chief series of $G$  does not depend on the series. Now, 
 we denote by  $L_A$  the  \emph{monolithic primitive group  associated to $A$},  
 that is 
$$L_{A}=
  \begin{cases}
   A\rtimes (G/C_G(A)) & \text{ if $A$ is abelian}, \\
   G/C_G(A)& \text{ otherwise}.
  \end{cases}
$$
If $A$ is a non-Frattini chief factor of $G$, then $L_A$ is a homomorphic image of $G$. 
 More precisely,  there exists 
 a normal subgroup $N$ such that $G/N \cong L_A$ and $\soc (G/N) \sim_G A$ (in the following we will sometimes identify $\soc L_A$ with  $A$ as $G$-groups). Consider now  all the normal subgroups $N$ with the property that  $G/N \cong L_A$ and $\soc (G/N) \sim_G A$: 
  the intersection $R_G(A)$ of all these subgroups has the property that  $G/R_G(A)$ is isomorphic to the crown-based  power $(L_A)_{\delta_G(A)}$  ($L_{A,  \delta_G(A)}$ for short).    
The socle $I_G(A)/R_G(A)$ of $G/R_G(A)$ is called the $A$-crown of $G$ and it is  a direct product of $\delta_G(A)$ minimal normal subgroups $G$-equivalent to $A$.  Later we will use the facts that 
$$I_G(A)=\{g \in G \mid g \textrm{ induces an inner automorphism on } A\}$$
 and $A \sim_G B$ implies $I_G(A)=I_G(B)$. In particular, if $A$ and  $B$ are chief factors of $G$ and $A \sim_G B$, then   $R_G(A)=R_G(B)$ and $L_A \cong L_B$.  

Note that if $L_k$ is a homomorphic image of $G$ for some $k\geq 1$  then  $L$ is  associated to a non-Frattini chief factor $A$ of $G$ ($L \cong L_A$) and $k \leq \delta_G(A)$. 
 If $L_{A,k}$ is a generating crown-based power  then   $L_{A,  \delta_G(A)}$ has the same property: in this case, by abuse of notation,  we will say 
that  $A$  is a \emph{generating chief factor} for $G$. 

The minimal number  of generators of a generating crown-based power can be computed when $A$ is abelian with the help of the following formula:
 for an irreducible $G$-module $M$,  set 
$$r_G(M)=\dim_{\End_G(M)}M  \quad \quad s_G(M)=\dim_{\End_G(M)} H^1(G,M)$$
 and define

$$h_{G}(M)= 
  \begin{cases}
  &\delta_G(M) \quad \text{ if $M$ is a  trivial } G\textrm{-module}, \\
 &  \left[\frac{s_G(M)-1}{r_G(M)}\right]+2
 \quad \text{ otherwise}.
  \end{cases}
$$

Note that, as $G/R \cong L_{M, k}$ where $R=R_G(M)$ and $k=\delta_G(M)$, we have 
 $\delta_G(M)=\delta_{G/R}(M)=\delta_{L_{M, k}}(M)$. Moreover,  if $\delta_G(M) >0$, then $R\leq C_G(M)$ and 
$\dim_{\End_G(M)} H^1(G,M)= \delta_G(M)+ \dim_{\End_G(M)} H^1(G/C_G(M),M) $
 (see e.g. [1.2] in \cite{andrea2}) and therefore 
$r_G(M)=r_{G/R}(M)$ and $s_G(M)=s_{G/R}(M).$ 
 We conclude that   if $\delta_G(M) >0$, then 
\begin{equation}\label{R}
h_G(M)= 
h_{ L_{M, \delta_G(M)}}(M).
\end{equation}
 
>From a result by Gasch\"utz \cite[Satz 2]{g2}, we have that  either $h_G(M)=d(L_{M,\delta_G(M)})$  or $h_G(M) < d(L_M/M).$
Therefore we have the following: 
\begin{prop}
If there exists  an abelian generating chief factor of $A$ of $G$, 
 then 
\[ d(G)=h_G(A).\]
\end{prop}

In our discussion we will employ different arguments according to the existence or not of an abelian generating chief factor. 
In the first case it is useful to notice that 
\begin{prop}\label{ab-crown}
Let  $d(I_G)$ be the minimal number  of generators of the augmentation ideal of $\Z G$ as a $G$-module. If $G$ has an abelian generating chief factor  $A$, then 
\[d(G)=d(I_G)=h_G(A).\]
\end{prop}
\begin{proof}
 By a result of Cossey, Gruenberg and Kov\'acs \cite[Theorem 3]{pr} 
$$d(I_G)= \max \{ h_G(M)\mid  M \ \textrm{irreducible }  G\textrm{-module} \},$$ thus $d(I_G) \geq h_G(A) =d(G).$ Since $d(I_G)\leq d(G)$, we have an equality. 
\end{proof}

  Theorem \ref{main} will be derived by an extension to the non-abelian crowns of the  
 following:  
\begin{prop}[Proposition 1 \cite{andrea2}]\label{a2-prop1}
If $H$ is a finite group and $G$ is a transitive permutation group of degree $n$, then 
 \[ d(I_{H \wr G}) = \max  \left\{ d(I_{H/H'\wr G}), \left[ \frac{d(I_H)-2}{n} \right] +2 \right\}. \]
\end{prop}


\section{Crowns in wreath products}

Let $H$ be a finite group and $K$ a transitive group of degree $n$
and  denote by $W=H \wr K= H^n \rtimes K$  the (permutational) wreath product of $H$  and  $K$, 
where $K$ permutes the components of the base subgroup $H^n= H_1 \times \cdots \times H_n$.  

In this section we want to study the relation between the chief factors of $H$ and the chief factors of $W$. First note that if $A$ is an $H$-group then $A^n$ can be seen as a $W$-group where $H^n$ acts componentwise and $K$ permutes the components of the elements. When dealing with $A^n$ as a $W$-group we will usually refer to this action. We say that an $H$-group $A$ is irreducible if the only
 $H$-groups contained in $A$ are $A$ and $\{ 1\}$; we say that an $H$-group is trivial if the action of $H$ on $A$ is the trivial one, that is $H=C_H(A)$.

 \begin{prop}\label{H-groups} 
Let $A$ and $B$ be irreducible $H$-groups. 
\begin{enumerate}
	\item If $A$ is a non-trivial  $H$-group, then $A^n$  is an irreducible non-trivial  $W$-group. 
	\item If $A \sim_H B$ then $A^n \sim_W B^n$. 
	\item  If $A$ and $B$ are non-trivial  $H$ groups and $A \nsim_H B$, then $A^n \nsim_W B^n$. 
\item If $A$ is a non-central chief factor of $H$ and $L$ is the associated monolithic group, then $A^n$ is a chief factor of $W$ and the monolithic primitive group associated to $A^n$ is isomorphic to $ L \wr K$. 
\end{enumerate}
\end{prop}
\begin{proof}
\begin{enumerate}
\item 
Let $N \neq 1$ a $W$-group contained in $A^n=A_1 \times \cdots \times A_n$ and let $1 \neq (x_1, \ldots , x_n) \in N $ be a non trivial element. As $K$ is transitive on the components, we can assume $x_1 \neq 1$. 
 Note that $C_A(H)$ is a proper $H$-subgroup of $A$, hence $C_A(H)=1$ by irreducibility of $A$. 
Thus $[x_1, H] \neq 1$ and in particular  $[x_1, H]$ is a non-trivial $H$-subgroup of $A$,  hence   $[x_1, H]=A.$ Therefore $[  (x_1, \ldots , x_n), H_1] =[x_1,H] \times \{1\} \times \cdots \times \{1\} =A_1$ is contained in $N$ and, by the transitivity of the action of $K$, we conclude that $A^n \leq N$. 
\item 
Let $A \sim_H B$: there exists an isomorphism 
 $\Phi: A\rtimes H \rightarrow B\rtimes H$ such that the following diagram commutes:

\begin{equation}\label{H}
\begin{CD}
1@>>>A@>>>A\rtimes H@>>>H@>>>1\\
@. @VV{\phi}V @VV{\Phi}V @|\\
1@>>>B@>>>B\rtimes H@>>>H@>>>1
\end{CD}
\end{equation}

\

Now 
define $\Psi: A^n\rtimes W\rightarrow B^n\rtimes W$ by the position
$$((a_1, \ldots, a_n) (h_1, \ldots, h_n) k)^\Psi  =  (a_1^\phi , \ldots, a_n^\phi ) (h_1^\Phi , \ldots, h_n^\Phi ) k .$$ 
 Thus $\Psi$ is a well defined isomorphism for which the following diagram is commutative:

\begin{equation}\label{W}
\begin{CD}
1@>>>A^n@>>>A^n\rtimes W@>>>W@>>>1\\
@. @VV{\psi}V @VV{\Psi}V @|\\
1@>>>B^n@>>>B^n\rtimes W@>>>W@>>>1
\end{CD}
\end{equation}
\
where $\psi$ is the restriction to $A^n$ of $\Psi$, 
 and therefore $A^n \sim_W B^n$. 
 
\item Assume, by contradiction, that $A^n \sim_W B^n$. 
We first consider the case where $A$ and $B$ are abelian. 
 Then the $W$-equivalence relation is simply  the $W$-isomorphism relation  and $A^n \sim_W B^n$
 implies that there exists a $W$-isomorphism $\psi: A^n \rightarrow B^n$. Note that $C_{A^n}(K)= \diag(A^n) \cong A$ and 
 similarly $C_{B^n}(K)= \diag(B^n) \cong B$. Since $\psi$ is a $W$-isomorphism,
 the restriction of $\psi$ to $C_{A^n}(K)$ is a $W$-isomorphism between $C_{A^n}(K)=\diag(A^n)$ and $C_{B^n}(K)=\diag(B^n)$. 
 This implies that there is 
  an $H$-isomorphism between $A$ and $B$, and we conclude that $A \sim_H B$. 
 
 We now consider the case where $A$ and $B$ are non-abelian. Assume that the  diagram  (\ref{W}) is commutative.  
  First of all we note that the minimal normal subgroups of $A^n\rtimes H^n$ contained in $A^n$ are the subgroups $A_i$. 
 Moreover the $A_i^\psi$ are minimal normal subgroups of $(A^n\rtimes H^n)^\Psi =B^n\rtimes H^n$ contained in $(A^n)^\psi=B^n$. It follows that  $A_i^\psi=B_j$
 for some $j$. In particular $A \cong B$ as groups.   
 
 If $A_1^\psi =B_1$, then consider that 
 $[\prod_{i>1}A_i, H_1]=1$ implies 
 $$\left[\prod_{i>1}A_i, H_1\right]^\Psi=\left[\prod_{i>1}A_i^\Psi, H_1^\Psi\right]=\left[\prod_{i>1}B_i, H_1^\Psi\right]=  1$$ 
 thus $H_1^\Psi \leq C_{B^n\rtimes H^n}(\prod_{i>1}B_i)
$. Moreover,   $H_1 ^\Psi \leq B^n\rtimes H_1$ since the right part of the diagram (\ref{W}) commutes, and therefore 
 $$H_1^\Psi \leq C_{B^n\rtimes H^n}\left(\prod_{i>1}B_i\right)\cap  \left( B^n\rtimes H_1 \right)\leq  B_1\rtimes H_1 .$$
 It follows that the following diagram commutes   
 \begin{equation*}
\begin{CD}
1@>>>A_1@>>>A_1\rtimes H_1@>>>H_1@>>>1\\
@. @VV{\psi}V @VV{\Psi}V @|\\
1@>>>B_1@>>>B_1\rtimes H_1@>>>H_1@>>>1
\end{CD}
\end{equation*}
  and $A_1 \sim _{H_1} B_1$.  Since the action of $H$ on $A$ and $B$ is equal to the action of $H_1$ on $A_1$ and $B_1$ respectively, $A\sim _H B$ and we are done. 
 
 We are left with the case $A_1^\psi \neq B_1$; then there exists $j \neq 1$ such that $A_j^\psi=B_1$.
 Note that we can not argue as above, since now 
 $A_1^\psi \rtimes H_1^\Psi$ is contained in $ B_1 B_j \rtimes H_1 $ 
but not in   $  B_j \rtimes H_1 $
 and hence  we can not simply ``restrict'' the diagram (\ref{W}) to one component.

Since the right part of the diagram (\ref{W}) commutes, for every $h \in H_1 $ there exist unique elements $b_i \in B$ such that   $h^\Psi=(b_1, \ldots, b_n)h$: 
 we define the map 
$\beta: H_1 \mapsto B_1$ by sending $h $ to the element $h^\beta=(b_1, 1 \ldots, 1)$. 
 Then  $[H_1,A_j]=1$ implies $[H_1^\Psi , B_1]=1$ and
hence $h^\beta h $ commutes with  every element of $ B_1$.
It  follows that the map $\Theta: A_j \rtimes H_1 \mapsto B_1 \rtimes H_1$ defined by $( a_j h)^\Theta =a_j^\psi h^\beta h$ is a well defined homomorphism for  which the following diagram is commutative
 \begin{equation*}
\begin{CD}
1@>>>A_j@>>>A_j\rtimes H_1@>>>H_1@>>>1\\
@. @VV{\psi}V @VV{\Theta}V @|\\
1@>>>B_1@>>>B_1\rtimes H_1@>>>H_1@>>>1
\end{CD}
\end{equation*}
 and hence $A_j \sim_{H_1}B_1$ 
 (note that the action of $H_1$ on $A_j$ is the trivial one and it is not equivalent to the action of $H$ on $A$).

Now, by definition,
 $$I_{H_1}(A_j) =\{x \in H_1 \mid x \textrm{ induces an inner automorphism on } A_j\}  =H_1,$$
 hence $A_j \sim_{H_1}B_1$ implies
 $I_{H_1}(B_1)=I_{H_1}(A_j) =H_1$. 
Then $I_W(B^n)=(I_H(B))^n=H^n$, and since $B^n \sim_W A^n$, we get 
$I_W(A^n)=I_w(B^n)=H^n$. Therefore $I_H(A)=H=I_H(B)$. As we will see in the subsequent Lemma \ref{banale}, 
 from  the facts that  $I_H(A)=H=I_H(B)$ and that $A \cong B$ as groups,  we get that $A$ and $B$ are $H$-equivalent to the same trivial $H$-group. 
  By transitivity, it follows that  $A \sim_H B$ and this  gives the desired contradiction.

\item Let $A$ be a chief factor of $H$. Then $L \cong H/C_H(A)$ if $A$ is non-abelian, $L \cong A \rtimes H/C_H(A)$ otherwise. Note that $C_W(A^n) \leq \cap_{i=1}^n C_W(A_i) \leq H^n$, as the action of $K$ on the components is faithful. 
 Hence $C_W(A^n)=C_H(A)^n.$ Then $W/C_W(A^n) \cong (H/C_H(A)) \wr K$ and the result follows.  
\end{enumerate}

\end{proof}

\begin{lemma}\label{banale}
Let $A$ be a $G$-group with trivial center. If $I_G(A)=G$ then $A$ is $G$-equivalent to the trivial $G$-group $A^*$, where $A^* =A$ as a group.   
\end{lemma}

\begin{proof}
This is a consequence of the definition (see remark after Proposition 1.2 in \cite{paz}) and Theorem 11.4.10 in \cite{robinson}, but for  the readers' convenience, we will sketch a direct proof. 

Since $A$ has trivial center and  $I_G(A)=G$, there is a homomorphism $f: G \mapsto A$ which send $g \in G$ to the element $f(g)$ in $A$ 
 such that  $a^{f(g)}=a^g$ for every $a \in A$. 
Let $A^*$ be the trivial $G$-group  equal to $A$ as a group. 
Now we define 
\begin{eqnarray*}
\Phi : &A^* \times G &\mapsto A \rtimes G\\
&(a,g) &\mapsto a f(g)^{-1}g.    
\end{eqnarray*}
Note that, by definition of $f$,  for every 
 $g \in G$ 
 the element $f(g)^{-1}g$ centralizes the elements of $A$ in  $A \rtimes G.
$ 
Thus $((a_1,g_1)(a_2,g_2))^\Phi= (a_1 a_2, g_1g_2)^\Phi=
a_1 a_2 f(g_1g_2)^{-1} g_1 g_2 
=a_1 (a_2 f(g_2)^{-1}) ( f(g_1)^{-1} g_1) g_2=
 a_1  (f(g_1)^{-1} g_1)  (a_2 f(g_2)^{-1})g_2=(a_1,g_1)^\Phi(a_2,g_2) ^\Phi$, since $a_2 f(g_2)^{-1} \in A$. This shows that $\Phi$ is a homomorphism.  
 Then the following diagram is commutative: 
\begin{equation*}
\begin{CD}
1@>>>A^*@>>>A^* \times G@>>>G@>>>1\\
@. @| @VV{\Phi}V @|\\
1@>>>A@>>>A \rtimes G@>>>G@>>>1
\end{CD}
\end{equation*}
and we conclude  that $A \sim_G A^*$. 
 \end{proof}

>From now on, $B$ will denote the base subgroup $H^n$ of $W=H \wr K =B \rtimes K$. 
Let us fix a chief series of $H$ passing through the derived subgroup $H'$ of $H$

\begin{equation}\label{chief-series of H} 
	1=N_t \lhd N_{t-1} \lhd  \cdots \lhd N_{t'}=H' \lhd \cdots \lhd N_1 \lhd N_0 =H.
\end{equation}

Since every  $N_i^n $ is normal in $W$, we can refine the series $(N_i^n )_i$ to get a  $W$-chief series of $B$  
 passing through the derived subgroup $B'$ 
\begin{equation}\label{chief-series of B} 
	1=M_{s_t} \lhd  \cdots  \lhd  M_{s_{t-1}} = N^n_{t-1} \lhd  \cdots \lhd
	 M_{s_{t'}}=N^n_{t'}=B' \lhd \cdots 
	  M_1 \lhd M_0 =B.
\end{equation}

For every prime $p$,  let $d_p({H/H'})$ be the minimal number of generators of the Sylow $p$-subgroup of ${H/H'}$. 
 Note that $d_p({H/H'})=h_{H/H'}(A)$ where $A$ is a   central non-Frattini  chief-factor of $H/H'$ of order $p$.  
Moreover,  if  $A=X/Y $ is a central non-Frattini (i.e. complemented) chief-factor of $H$, then $X$ can not be contained in $H'$; therefore  
\begin{equation}\label{central}
d_p({H/H'})=h_H(\F_p)=h_{H/H'}(\F_p)
\end{equation}
 where $A \sim_H \F_p$ and   $\F_p$ is the irreducible  trivial $\F_p H$-module. 

\begin{prop}\label{chief} 
Let $M=M_i/M_{i+1}$ be a non-Frattini chief factor   of the series \ref{chief-series of B}.
\begin{enumerate}
	\item If $M_i \leq B'$, then there exists a non-Frattini chief factor $A=X/Y$ of the series \ref{chief-series of H} contained in $H'$ such that $M =X ^n/Y^n$. Moreover $M$ is not  $W$-equivalent to any chief factor of $W/B'$, 
  $\delta_W(M)=\delta_H(A)$ and $L_M \cong L_A \wr K$.  
\item If  $B' \leq  M_{i+1} < M_i \leq B$,  then $\delta_W(M) \leq \delta_K(M) + d_p( H/H') r_K(M)$.  
	\item If $B \leq M_{i+1},$ and $M$ is not equivalent to any 
 $W$-chief factor of $B/B'$, 
 then the action of $W$ on $M$ induces an action of $K$ on $M$, $\delta_W(M)=\delta_K(M)$ and the primitive monolithic group associated to $M$ is the same in the two actions.  
\end{enumerate}
\end{prop}

\begin{proof} 
 \begin{enumerate}
	\item We first prove that   the map $A=X/Y \mapsto A^n=X^n/Y^n$ gives a bijection between the set of non-Frattini chief factors of the series \ref{chief-series of H} contained in $H'$   and the set of 
	non-Frattini chief factors of the series \ref{chief-series of B} contained in $B'$.

 Let $A=X/Y $ be a  non-Frattini chief factor of the series \ref{chief-series of H} contained in $H'$. 
Note that  the central complemented chief factors of  \ref{chief-series of H}  lie above $H'$.
 Then $A$ is not central and  hence,  
by  Proposition \ref{H-groups} , we have that $A^n$ is a   non-central chief factor of the series \ref{chief-series of B}     contained in $B'$. Moreover, if $U$ is a complement to $A$ in $H$, then $U \wr K$ is a complement to $A^n$ in $W$.  This implies that the map is well defined. 

To prove that the map is bijective, it is sufficient to  show that if $A=N_i/N_{i+1}$ is a Frattini chief factor of $H$,  then every chief factor $X/Y$ of the series \ref{chief-series of B} with 
 $N_{i+1}^n \leq Y < X \leq N_i^n$ is Frattini. We can assume $N_{i+1}=1$; thus $A \leq \frat H$   and $A^n \leq (\frat H )^n= \frat B \leq \frat W$ and we are done.

To prove that $\delta_H(A)=\delta_W(A^n)$ it is sufficient to show  that $A^n$ can not be equivalent to any $W$-chief factor containing $B'$, indeed, from  Proposition \ref{H-groups}, we already know there are $\delta_H(A)$ chief factors of \ref{chief-series of B} $W$-equivalent to $A^n$ inside $B'$. 
 Assume, by contradiction, that $A^n \sim_W M =X/Y$ where $B' \leq Y \leq X \leq W$. Then $I_W(A^n)=I_W(M)$. But, on one hand, $I_W(A^n) = (I_H(A))^n \leq B$, on the other hand $I_W(M)=X C_W(X)$. This implies that $X \leq B$ and $I_W(M)=B$. In particular, as $B' \leq Y$,  $M$ is  centralized  by $B$. Therefore the two factors $A^n$ and $M$ are abelian, the equivalence relation reduces to a $W$-isomorphism  and hence $A^n$ is centralized by $B$. 
It follows that $A$ is a central factor of $H$, but this is a contradiction, since complemented central  chief factors of  \ref{chief-series of H}  lie above $H'$.   

Finally, by Proposition \ref{H-groups} we get that $L_M = L_A \wr K$.  

	\item Set ${\ol H}=H/H'$ and note that $B \leq C_W(M)$, hence the action of $W$ on $M$ induces an action of $K$ on $M$. We follow the arguments of Lemma 2.1 \cite{andrea1} and Lemma 4.1 in \cite{andrea2}.  Since we are dealing with non-Frattini factors,  we can assume that the Frattini subgroup of ${\ol H}$ is trivial.  The Sylow $p$-subgroup  ${\ol H}_p$ of ${\ol H}$ is a vector space of dimension
 $d=d_p({\ol H})$ generated, let say, by the elements $h_1, \ldots , h_d$. Then the  Sylow $p$-subgroup ${\ol H}_p^n$  of ${\ol H}^n$ is generated, as an $\F_pK$-module, by the elements $(h_i,\ 1, \ldots ,\ 1)$. In particular  ${\ol H}^n$  is the direct sum of $d$ cyclic  $\F_pK$-modules, and the number of complemented  $\F_pK$-modules $K$-equivalent to $M$ in ${\ol H}^n$ is at most  $d_p({\ol H}) r_K(M)$ where $ r_K(M)=\dim_{\End_K(M)}(M)$ (see Lemma 2.1 \cite{andrea1}).
 It follows that $\delta_W(M) \leq \delta_K(M) + d_p(\ol H) r_K(M)$.

\item It is sufficient to note that $B \leq C_W(M)$ and that, by the first part of the proposition, $M$ can not be equivalent to any chief factor contained in $B'$.   
\end{enumerate}
\end{proof}

Now  we consider non-trivial $W$-modules (abelian $W$-groups) and the values of the function $h_W$ on them. 

\begin{prop}\label{moduli} 
Let $p$ be a prime and  $M$ be a non-trivial irreducible   $\F_p W$-module. 

\begin{enumerate}
	\item If $M$ is $W$-equivalent to a non-Frattini $W$-chief factor contained in $B'$, then
 there exists a non-trivial irreducible   $\F_p H$-module $U$ such that $M \sim_W U^n$ and 
 $h_W(M) \leq \left\lceil \frac{h_H(U)-2}{n}\right\rceil+2$. 
	\item If $M$ is $W$-equivalent to a non-Frattini $W$-chief factor of $B/B'$, then
 $h_W(M)\leq h_K(M)+d_p(H/H')$. 
	\item If $M$ is not $W$-equivalent to any non-Frattini $W$-chief factor of $B$ but  $\delta_W(M)=\delta_K(M)\geq 1$, then  
	$h_W(M)=h_K(M)$. 
	\item If $\delta_W(M)=0$, then $h_W(M) \leq 2$. 
\end{enumerate}
\end{prop}

\begin{proof}
\begin{enumerate}
	\item The first part follows from 
 Propositions \ref{chief}, 
 the bound of $h_W(M)$ is proved in 
    \cite[step 2.5]{andrea2}. 
	\item 
Since  $M$ is $W$-equivalent to a chief factor of $B/B'$, 
 $B$ centralizes $M$ and hence $r_W(M)=r_K(M)$.
 Let  ${\ol H}=H/H'$.  
By Proposition \ref{chief},   $\delta_W(M) \leq \delta_K(M) + d_p(\ol H) r_K(M)$. 
 Moreover (see (1.2) in \cite{andrea2}) 
\begin{eqnarray*}
	s_W(M)&=&\delta_W(M) +\dim_{\End_W(M)} H^1(W/C_W(M),M) \\
	&\leq&  \delta_K(M) + d_p(\ol H) r_K(M) +\dim_{\End_K(M)} H^1(K/C_K(M),M)\\
	&=& d_p(\ol H) r_K(M)+s_K(M).
\end{eqnarray*}
Therefore, 
\begin{eqnarray*}
h_W(M)&=& \left[\frac{s_W(M)-1}{r_W(M)}\right]+2 \\
& \leq &
  \left[\frac{ d_p(\ol H) r_K(M)+s_K(M)-1}{r_K(M)}\right]+2 \\
 & \leq & h_K(M)+d_p(\ol H).
\end{eqnarray*}
	\item Since  $\delta_W(M)=\delta_K(M)\geq 1$, we have that $M$ is not equivalent to any chief factor contained in $B$  and hence 
 $B$ is contained in $R_W(A)$ where  $A$ is  a chief factor  $W$-equivalent to  $M$ 
 (every minimal normal subgroup of $W/R_W(A)$ is $W$-equivalent to $A$). 
 By the same arguments used to prove  equation \ref{R}, it follows that  $h_W(M)=h_{W/B}(M)=h_K(M)$.
 \item This is proved in  Lemma 1.5 of \cite{andrea-h}. 
\end{enumerate}
\end{proof}

\section{Number of generators of wreath products} 

Let $L$ be a monolithic primitive group with socle $N$. 
Let us denote by $P_L(d)$
 (resp. $P_{L/N}(d))$  the probability of generating
 $L$ (resp. $L/N$) with $d$ elements, and, for $d \geq d(L)$,  let 
 \[P_{L,N}(d)=P_L(d)/P_{L/N}(d).\]

When $N$ is non-abelian, the formula given in \cite{dv-l} to evaluate $d(L_t)$ is the following: 

\begin{thm}\cite[Theorem 2.7]{dv-l}\label{f}
Let $L$ be a monolithic primitive group with  non-abelian socle $N$ and let $d \geq d(L)$. 
Then $d(L_t) \leq d $ if and only if 
\[t \leq \frac{P_{L,N}(d) |N|^d}{|C_{\aut L}(L/N)|}.\]
\end{thm}

In  Theorem 1.1 in \cite{fiore}  it is proved that if $|N|$ is large enough  and $d\geq 2$ random elements generate $L$ modulo $N$,  
 then these elements almost certainly generate $L$ itself: 

\begin{thm} \cite[Theorem 1.1]{fiore}\label{fio}
There exists a positive integer $k_0$ such that, if $L$ is a  monolithic primitive group with socle $N$ and  $|N|\geq k_0$, then for every $d \geq d(L)$ we have $P_{L,N}(d) \geq 1/2$.  
\end{thm}

\begin{prop}\label{L_t}
Let $L$ be a  monolithic primitive group with a non-abelian socle $N$,  $K$ a transitive group of degree $n$
and  $L^*=L \wr K$. 
Assume that $|N|^n \geq k_0$. For every positive integer $t$ and every integer  $d \geq d(L^*/\soc{L^*})-2 $, if $d(L_t)\leq d \cdot n$, then
 $d(L^*_t)\leq d+2$.  
\end{prop}

\begin{proof} 
Since $L_t$ can be generated by $nd$ elements, by Theorem \ref{f}
 we have that 
 $$t \leq  \frac{P_{L,N}(nd) |N|^{nd}}{|C_{\aut L}(L/N)|}.$$
 As  $N \leq C_{\aut L}(L/N)$ and $P_{L,N}(nd) \leq 1$, we deduce $t \leq |N|^{nd -1}$.

Now, again by Theorem \ref{f}, to prove that $d(L^*_t) \leq d+2$, it is sufficient to prove that 
$$t \leq    \frac{P_{L^*,M}(d+2) |M|^{d+2}}{|C^*|}$$ where $M=\soc L^*$ and  $C^*=C_{\aut L^*}(L^*/M)$. By assumption $d+2 \geq  \max (d(L^*/M), 2)=d(L^*)$, where the last equation follows from \cite{meneg},   and $|M|=|N|^n \geq k_0$. 
Thus we can apply Theorem \ref{fio} to get that $P_{L^*,M}(d+2) \geq 1/2$. Moreover, if $N=S^a$, where $S$ is a simple non-abelian group and $a$ a positive integer, from the proof of  Lemma 1 in \cite{dv-l-2},  $|C^*|\leq n a |S|^{na-1}|\aut S| \leq na|S|^{na+1}$. 
 It follows that   
 $$ \frac{P_{L^*,M}(d+2) |M|^{d+2}}{|C^*|}  \geq \frac{1}{2} \cdot
  \frac{ |M|^{d+2}}{ na|S|^{na+1}}  .$$
 Since $t \leq |N|^{nd -1}$ and $M=N^n$, it is sufficient to check that 
 $ \frac{ |N|^{n(d+2)}}{2 na|S|^{na+1}}  \geq  |N|^{nd -1}$,  that is 
$|N|^{2n+1}=|S|^{2na+a} \geq 2 na|S|^{na+1}$, and this  follows from the fact that  $|S| \geq 60$. 
 \end{proof}

\begin{prop}\label{wreath}
Let $K$ be a transitive permutation group of degree $n \geq \log_{60} k_0$, where $k_0$ is the constant defined  in  Theorem \ref{fio}. Then 
 $$d\left( H \wr K\right) \leq \max \left(d\left(H/H' \wr K\right),  \left\lceil \frac{d(H)}{n} \right\rceil +2\right).$$
\end{prop}

\begin{proof} 
Set $\ol H=H/H'$. 
When $W= H \wr K$ has an abelian generating chief factor, by Proposition \ref{ab-crown} $d(G)=d(I_G)$, and then  the result follows from 
 Proposition \ref{a2-prop1}:  
\begin{eqnarray*} 
d(W)&=&d(I_W)= \max \left( d\left(I_{\ol H \wr K}\right), \left[ \frac{d(I_H) -2}{n} \right]+2\right)\\
 &\leq & \max \left( d\left({\ol H \wr K}\right), \left[ \frac{d(H) }{n} \right]+2 \right).
\end{eqnarray*} 

Now we assume that every generating chief factor  is non-abelian and we argue by induction on $|H|$, the case $|H|=1 $ being obviously true. 
Let $M$ be a non-abelian generating chief factor of the series \ref{chief-series of B}. If $M$ is not contained in $B'$, then,   by Proposition \ref{chief}, $M$ is a  $K$-group such that   $\delta_W(M)=\delta_K(M)$ and the crown-based power $L_{M, \delta_W(M)}$ is a homomorphic image of $K$. 
Therefore 
$$d(W) =d\left(L_{M, \delta_W(M)}\right) \leq d(K) \leq d\left(\ol H \wr K\right)$$ and the result follows.

We are left with the case where $M$ is a non-abelian chief factor contained in $B'$. From Proposition \ref{chief} we know that there exists a non-abelian chief factor  $N$ of the series \ref{chief-series of H} such that $\delta_W(M)=\delta_H(N)$ and $L_M \cong L_N \wr K$. 
Set $L=L_N$, $L^*=L \wr K$ and $\delta=\delta_H(N)$.

%
%
Let $d_0=\max \left(d\left(\ol H \wr K\right),  \left\lceil \frac{d(H)}{n} \right\rceil +2\right)$; we want to apply Proposition \ref{L_t} to 
 prove that $d(W) =d\left(L^*_\delta\right)\leq d_0$.  
 As  $|L/N|<|H|$, by induction we get 
 $$ d\left( L/N \wr K\right) \leq  \max \left(d\left({L/L'} \wr K\right),  \left\lceil \frac{d(L/N)}{n} \right\rceil +2\right).$$ Since  
  ${L/L'}$  is  a homomorphic image of $\ol H$ and $L^*/M= L/N \wr K$, we deduce 
 that 
 $$d\left(L^*/M\right)=d\left( L/N \wr K\right)  \leq \max \left(d\left(\ol H \wr K\right),  \left\lceil \frac{d(H)}{n} \right\rceil +2\right)= d_0.$$ 
 Moreover $d_0 \geq  \left\lceil \frac{d(H)}{n} \right\rceil +2
$,  that is $n(d_0-2) \geq d(H) \geq d\left(L_\delta\right)$. Also, the assumption  $n \geq \log_{60} k_0$,  gives $|N|^n \geq k_0$.
 Therefore  all the hypothesis of Proposition  \ref{L_t} are satisfied (for $d=d_0-2$) and we conclude that  $d(W) =d\left(L^*_\delta\right)\leq d_0$.
 \end{proof}

The previous result reduces the problem of finding a bound to $d(W)$ to the case where $H$ is an abelian group.  
Let \[ \rho_{K,H,p}= \max_{M} h_K (M) +d_p\left( H/H'\right)\]
where $M$ ranges over the set of non trivial irreducible $\F_pK$-modules, with $\rho_{K,H,p}= 0$ if every 
irreducible $\F_pK$-module is trivial.

\begin{prop}\label{wreath-ab} 
 If $H$ is abelian, then  
 $d( H \wr K) \leq \max_{p \mid | H|} ( d(  H \times K), \ \rho_{K,H,p}).$
 \end{prop}
\begin{proof}
Let $W=H \wr K$ and let $M$ be a generating chief factor for $W$. 

If $M$ is non-abelian, then $M$ can not be $W$-equivalent to any chief factor of $B=H^n$, 
 hence  $R_W(M) \geq B$ and $L_{M, \delta_W(M)}$ is a homomorphic image of $K$. 
 It follows that 
$$d(W)=d(L_{M, \delta_W(M)}) \leq d(K) \leq d(  H \times K)$$ 
and we are done. 

Now, let us assume that $M$ is abelian.
If  $M$ is central, by equation \ref{central}  it follows that $h_W(M)= h_{W/W'}(M) \leq d(W/W') \leq d(H \times K)$ 
 since  $W/[B,K] \cong H \times K$. 
Thus $d(W) = h_W(M) \leq d(H \times K)$ and the result follows. 

Then we are left with the case where $M$ is  non-central. By Proposition   \ref{moduli} (both (2) and (3)),
 $ h_{  W}(M)  \leq h_K(M)+d_p(H)$ and therefore $d(W)= h_{  W}(M)  \leq  \rho_{K,H,p}$. This completes the proof. 
\end{proof}

\section{Iterated Wreath products} 

Note that if $K$ is a permutation group of degree $n$, then 
$$d(H) \leq n \cdot d( H \wr K);$$ indeed,   given a  set 
$$\{g_i=(h_{i,1}, \ldots, h_{i,n})k_i \mid \ h_{i,j} \in H, \ k_i \in K, \ i=1,\ldots ,d \}$$
 of generators for $H \wr K$,  
 then $H$ can be generated by the elements $\{ h_{i,j} \mid \ j=1, \ldots ,n \textrm{, } i=1,\ldots ,d\}$.  Moreover, 
 $$d(H \wr K) \geq
d( H/H' \times K/K')  $$ since $H/H' \times K/K'$ is a homomorphic image of 
 $ H \wr K $. 
 
 This shows the \lq\lq only if" implication of Theorem \ref{main}. The other implication is proved in the following theorem.   
\begin{thm}
Let $(G_i)_{i\in \N}$  be a sequence of  transitive permutation  groups of degree $n_i$. Let $\ol G_i=G_i/G'_i$ and denote by $W_m= G_m \wr \cdots \wr G_1$ the iterated permutational wreath product of the first $m$ groups. 
 Assume  that there exists two integers $\C$ and $\D$ with    
\begin{itemize}
\item[(i)] $d\left(\prod_{i=1}^{\infty} \ol G_i\right)=\C$
\item[(ii)] $d(G_i) \leq \D \cdot n_1 \cdots n_{i-1}$ for every $i>1$.  
\end{itemize} 
Then, for $\E=\max (\D+2, d(W_{i_0}))$, where $i_0$ is the first index such that 
the degree $ n_1 \cdots n_{i_0}$ of $W_{i_0}$ is at least $\log_{60}(k_0)$,  we get the following: 
\begin{enumerate}
\item If $M$ is a non-trivial irreducible  $\F_p W_m$-module, where $m \geq i_0$, then  
\[h_{W_m}(M) \leq \E + d_p  \left(\prod_{i=i_0}^{m} \ol G_i\right);\]
\item $d(W_m) \leq \E+ d\left(\prod_{i=i_0}^{m} \ol G_i\right)$ for every $m \geq i_0$; 
\item  The inverse limit of the iterated wreath products $W_m$ is finitely generated and $d\left(\varprojlim_m W_m\right) \leq \E+\C.$ 
\end{enumerate}
\end{thm}

\begin{proof} 
\begin{enumerate}
\item
We argue by induction on $m$. The case $m = i_0$, is trivial since 
$h_{W_{i_0}}(M) \leq d(W_{i_0}) \leq \E$.  
 So let $m >i_0$ and let $M$ be  a non-trivial irreducible  $\F_p W_m$-module.
 By Proposition \ref{moduli} applied to  $W_m=G_m \wr_n W_{m-1}$, where   $n=n_1 \cdots n_{m-1}$ is the degree of $W_{m-1}$, 
 we get that either $h_{W_m}(M) \leq \left\lceil \frac{h_{G_m}(U)-2}{n}\right\rceil+2$
  for an $\F_p G_m$-module $U$ contained in $G'_m$, or  $h_{W_m}(M)\leq h_{W_{m-1}}(M)+d_p(\ol G_m)$; thus 
\[h_{W_m}(M) \leq \max \left(  \left\lceil \frac{h_{G_m}(U)-2}{n}\right\rceil+2,
  h_{W_{m-1}}(M)+d_p\left( \ol G_m\right) \right).\]
Since $h_{G_m}(U) \leq d(G_m)\leq \D n$ implies   $\left\lceil \frac{h_{G_m}(U)-2}{n}\right\rceil+2 \leq \D+2$, and, by inductive hypothesis $ h_{W_{m-1}}(M) \leq 
 \E + d_p \left(\prod_{i=i_0}^{m-1} \ol G_i\right)$, we get 
 \[h_{W_m}(M) \leq \max \left(  \D+2 ,
 \E + d_p \left(\prod_{i=i_0}^{m-1} \ol G_i\right) +d_p\left( \ol G_m\right) \right) \leq \E + d_p\left(\prod_{i=i_0}^{m} \ol G_i\right) .\]

\item Again, we argue by induction on $m$, the case  
 $m = i_0$ being trivial.  

So  let  $m > i_0$ that is  $n=n_1 \cdots n_{m-1}> \log_{60}(k_0)$.
   Proposition \ref{wreath} applied to  $W_m=G_m \wr_n W_{m-1}$, gives  
   \begin{eqnarray} 
 d(W_m) &\leq& \max \left(d\left(\ol G_m  \wr  W_{m-1}\right), 
 \left\lceil \frac{d(G_m)}{n} \right\rceil +2\right) \nonumber \\ 
&\leq& \max \left(d\left(\ol G_m  \wr  W_{m-1}\right), \label{prima} 
 \D+2\right).
 \end{eqnarray} 

Then we apply Proposition  \ref{wreath-ab} 
to  have 
\begin{eqnarray} \label{1}
 d(\ol G_m  \wr  W_{m-1}) \leq
 \max_{p \mid |\ol G_m |} \left( d\left( \ol G_m \times W_{m-1}\right), \ \rho_{W_{m-1},G_m,p}\right)
 \end{eqnarray} 
where 
\[ \rho_{W_{m-1},G_m,p}= \max_{M}\left( h_{W_{m-1}} (M)\right) +d_p\left(\ol G_m \right) \]
and  $M$ ranges over the set of non trivial irreducible $\F_pW_{m-1}$-modules, with $\rho_{W_{m-1},G_m,p}= 0$ if every 
irreducible $\F_pW_{m-1}$-module is trivial. By part (1) of this theorem, 
$  h_{W_{m-1}} (M) \leq \E + d_p \left(\prod_{i=i_0}^{m-1} \ol G_i\right)$, 
 and hence 
\begin{eqnarray}\label{2}
 \rho_{W_{m-1},G_m,p} \leq  \E + d_p \left(\prod_{i=i_0}^{m-1} \ol G_i\right) +d_p\left(\ol G_m \right)
=  \E + d_p \left(\prod_{i=i_0}^{m} \ol G_i\right).
 \end{eqnarray} 

Moreover, note that a crown-based power homomorphic image of $ \ol G_m \times W_{m-1}$ is either a homomorphic image of  $ W_{m-1}$ or  a homomorphic image of $ \ol G_m \times \ol W_{m-1}$ (in the latter case it is associated to a central chief factor). This implies that  
\begin{eqnarray*}
 d\left( \ol G_m \times W_{m-1}\right) &\leq& \max \left(   d\left( \ol G_m \times \ol W_{m-1}\right), d \left( W_{m-1}\right) \right) \\
 &\leq &  \max  \left(   d\left(\prod_{i=i_0}^{m} \ol G_i\right), d \left( W_{m-1}\right)\right).
 \end{eqnarray*} 
 By inductive hypothesis we get $ d \left( W_{m-1}\right) \leq  \E+ d\left(\prod_{i=i_0}^{m-1} \ol G_i\right)$, and 
 therefore 
\begin{eqnarray}
 d\left( \ol G_m \times W_{m-1}\right) 
 &\leq &  \max  \left(   d\left(\prod_{i=i_0}^{m} \ol G_i\right), \E+ d\left(\prod_{i=i_0}^{m-1} \ol G_i\right)\right) \nonumber \\
 &\leq&  \E+ d \left(\prod_{i=i_0}^{m} \ol G_i\right). \label{3}
 \end{eqnarray}  

>From (\ref{1}), (\ref{2}) and  (\ref{3}), we obtain that 
\begin{eqnarray*}
d\left(\ol G_m  \wr  W_{m-1}\right) &\leq&
 \max_{p \mid |\ol G_m |}  \left( d\left( \ol G_m \times W_{m-1}\right), \ \rho_{W_{m-1},G_m,p}\right)\\
&\leq & \max_{p \mid |\ol G_m |}  \left( \E+ d\left(\prod_{i=i_0}^{m} \ol G_i\right),  \E + d_p \left(\prod_{i=i_0}^{m} \ol G_i\right)\right)\\
&\leq & \E+ d\left(\prod_{i=i_0}^{m} \ol G_i\right). 
 \end{eqnarray*}  
Since $\D+2 \leq \E$, from (\ref{prima})  we conclude that 
\begin{eqnarray*}
d\left(W_m\right) &\leq& \max \left(d\left(\ol G_m  \wr  W_{m-1}\right), \D+2\right)\\
 &\leq& \E+ d\left(\prod_{i=i_0}^{m} \ol G_i\right).
 \end{eqnarray*}  

\item This follows directly from (2) and the assumption that  $d\left(\prod_{i=1}^{\infty} \ol G_i\right)=\C$.  
Indeed $d\left(W_m\right) \leq  \E+ d\left(\prod_{i=i_0}^{m} \ol G_i\right) \leq \E+\C$ for every $m$, and the same bound applies to the generating number of their inverse limit. 
\end{enumerate}
\end{proof}


\section{Probability of generating an iterated wreath product} 

Once we know that a profinite group $G$ is finitely generated, it is natural to ask about the probability to find a set of generators for the group. A profinite group $G$ is called  Positively Finitely Generated (PFG) if there exists an integer $t \geq d(G)$ such that a randomly chosen $t$-tuple generates $G$ with positive probability. 

Note that it is possible to extend the definitions of $G$-equivalence and crowns to profinite groups (see \cite{corone-profin}). Moreover, if $G$ is finitely generated then $\delta_G(A)$ is finite for every finite irreducible $G$-group $A$ and in particular  this holds for the chief factors of $G$ \cite[Theorem 12]{corone-profin}. 
 Recently, Jaikin-Zapirain and Pyber gave a characterization of PFG-groups in terms of non-abelian crowns: 

\begin{thm}[Jaikin-Zapirain, Pyber \cite{pfg-j-p}]\label{pfg1}
A finitely generated profinite group $G$ is PFG if and only if there exists a constant $c$ such that for every non-abelian chief factor $A$ of $G$, $\delta_G(A) \leq l(A)^c$ where $l(A)$ is the minimal degree of a faithful transitive representation of $A$. 
\end{thm}

This allows us to characterize PFG infinitely  iterated permutational wreath products. 
\begin{prop}\label{pfg}
Let $(G_i)_{i\in \N}$ be a sequence of  transitive permutation  groups of degree $n_i$.  
Assume that the inverse limit $W_\infty$ of the iterated permutational wreath products 
  $W_m= G_m \wr \cdots \wr G_1$ is finitely generated.
 Then $W_\infty$ is PFG if and only if 
there exists a constant $c$ such that for every non-abelian chief factor $A$ of $G_i$ and for every $i > 1$,
  $\delta_{G_i}(A) \leq l(A)^{c n_1 \cdots n_{i-1}}$. 
\end{prop}

\begin{proof} 
Let $M$ be a non-abelian chief factor of $W=W_\infty$ such that $\delta_W(M)>0$. Since  $\delta_W(M)$ does not depend on the chosen chief series and is finite (Theorems 11 and 12 in \cite{corone-profin}), then  $\delta_W(M)=\delta_{W_i}(M)$ for some $i$; let $i$ be the smallest integer with this property. Without loss of generality we can assume $i>1$. 
 Since 
 $\delta_{W_{i-1}}(M)<\delta_{W_i}(M)$, $M$ is equivalent to a non-abelian chief factor of $B=G_i^n$, the base subgroup of $W_i =G_i \wr W_{i-1}$, where $n=n_1 \cdots n_{i-1}$ is the degree of $W_{i-1}$. In particular $M$    is equivalent to a non-abelian chief factor  contained in
 $B'$, and from Proposition \ref{chief}  it follows that  there exists a non-abelian chief factor $A$ of $G_i$ such that $M \sim_{W_i} A^n$ and $\delta_{W_i}(M)= \delta_{G_i}(A)$.  Since  $ l(M)=l(A)^{ n} $ (see Proposition 5.2.7 in \cite{kleidman-lib} and the comments afterwords), the result follows from   
 the  characterization of PFG-groups given by  Jaikin-Zapirain and Pyber (Proposition \ref{pfg1}). 
\end{proof}

\end{document}